\documentclass[12 pt]{amsart}%
\usepackage{graphicx, color}
\usepackage{amscd}
\usepackage{amsmath}
\usepackage{amsfonts}
\usepackage{amssymb}
\usepackage{graphicx}%
\providecommand{\U}[1]{\protect\rule{.1in}{.1in}}
\providecommand{\U}[1]{\protect\rule{.1in}{.1in}}
\theoremstyle{plain}

\newtheorem{theorem}{Theorem}[section]
\newtheorem{proposition}[theorem]{Proposition}

\newtheorem{remark}[theorem]{Remark}

\numberwithin{equation}{section}

\begin{document}
\title{A note on lineability }
\author{G. Botelho, D. Diniz, D. Pellegrino and E. Teixeira}
\address[Geraldo Botelho]{ Faculdade de Matem\'{a}tica, Universidade Federal de
Uberl\^{a}ndia, 38.400-902 - Uberl\^{a}ndia, Brazil, e-mail: botelho@ufu.br.\\
[Diogo Diniz] UAME-UFCG, Caixa Postal 10044, Cep 58109-970, Campina Grande,
PB, Brazil \\
[Daniel Pellegrino] Departamento de Matem\'{a}tica, Universidade Federal da
Para\'{\i}ba, 58.051-900 - Jo\~{a}o Pessoa, Brazil, e-mail:
dmpellegrino@gmail.com. \noindent\\
[Eduardo Teixeira] Univerisdade Federal do Cear\'{a}, Depto de Matem\'{a}tica,
Av. Humberto Monte, s/n, Fortaleza-CE, Brazil. CEP 60.455-760. }

\begin{abstract}
In this note we answer a question concerning lineability of the set of
non-absolutely summing operators.

\end{abstract}
\maketitle

\section{\bigskip Introduction and main result}

A subset $A$ of an infinite-dimensional vector space $V$ is $\mu$-lineable if
$A\cup\{0\}$ contains an infinite-dimensional subspace of dimension $\mu$. Let
$\aleph_{0}$ be the countable cardinality and $\aleph_{1}$ be the cardinality
of $\mathbb{R}$. From now on $E$ and $F$ denote Banach spaces, the space of
absolutely $p$-summing linear operators from $E$ to $F$ will be denoted by
$\Pi_{p}(E;F),$ the space of bounded linear operators from $E$ to $F$ will be
represented by $\mathcal{L}(E;F)$ and the space of compact operators from $E$
to $F$ is represented by $\mathcal{K}(E;F).$ For details on the theory of
absolutely summing operators we refer to \cite{Diestel}.

In recent papers \cite{jmaa, seo} it was shown that under certain
circumstances $\mathcal{L}(E;F)\diagdown\Pi_{p}(E;F)$ is $\aleph_{0}%
$-lineable. In \cite{jmaa} there is a question from the anonymous referee,
asking about the possibility of proving that the set is $\mu$-lineable, for
$\mu>\aleph_{0}.$ Our next result shows that an adaptation of the proof of
\cite{jmaa} answers this question in the positive:

\begin{theorem}
\label{aaa}Let $p\geq1$ and $E$ be superreflexive. If $E$ contains a
complemented infinite-dimensional subspace with unconditional basis or $F$
contains an infinite unconditional basic sequence then $\mathcal{K}%
(E;F)\diagdown\Pi_{p}(E;F)$ (hence $\mathcal{L}(E;F)\diagdown\Pi_{p}(E;F)$) is
$\aleph_{1}$-lineable.
\end{theorem}

\begin{proof}
Assume that $E$ contains a complemented infinite-dimensional subspace $E_{0}$
with unconditional basis $(e_{n})_{n=1}^{\infty}$. First consider%
\begin{equation}
\mathbb{N}=A_{1}\cup A_{2}\cup\cdots\label{bbbb}%
\end{equation}
a decomposition of $\mathbb{N}$ into infinitely many infinite pairwise
disjoint subsets $(A_{j})_{j=1}^{\infty}$. Since $\{e_{n}\,;\,n\in
\mathbb{N}\}$ is an unconditional basis, it is well known that $\{e_{n}%
\,;\,n\in A_{j}\}$ is an unconditional basic sequence for every $j\in
\mathbb{N}$. Let us denote by $E_{j}$ the closed span of $\{e_{n}\,;\,n\in
A_{j}\}.$ As a subspace of a superreflexive space, $E_{j}$ is superreflexive
as well, so from \cite[Theorem]{davis} it follows that for each $j$ there is
an operator%
\[
u_{j}\colon E_{j}\longrightarrow F
\]
belonging to ${\mathcal{K}}(E_{j};F)\diagdown\Pi_{p}(E_{j};F)$. From the proof
of \cite{jmaa} we know that each projection $P_{i}\colon E_{0}\longrightarrow
E_{i}$ is continuous and has norm $\leq\varrho$ (the constant of the
unconditional basis of $E_{0}$). This also implies that each $E_{i}$ is a
complemented subspace of $E_{0}$. If $\pi_{0}\colon E\longrightarrow E_{0}$
denotes the projection onto $E_{0}$, for each $j\in\mathbb{N}$ we can define
de operator%
\[
\widetilde{u_{j}}\colon E\longrightarrow F~,~\widetilde{u_{j}}:=u_{j}\circ
P_{j}\circ\pi_{0}.
\]
Since $(P_{j}\circ\pi_{0})(x)=x$ for every $x\in E_{j}$, it is plain that
$\widetilde{u_{j}}$ belongs to $\mathcal{K}(E;F)\diagdown\Pi_{p}(E;F)$. There
is no loss of generality in supposing $\left\Vert \widetilde{u_{j}}\right\Vert
=1$ for every $j$. Now, consider the map
\begin{align*}
T &  :\ell_{1}\rightarrow\mathcal{K}(E;F)\\
T((a_{n})_{n=1}^{\infty}) &  =%
{\displaystyle\sum\limits_{j=1}^{\infty}}
a_{j}\widetilde{u_{j}}%
\end{align*}
Since the supports of the $\widetilde{u_{n}}$ are disjoint it is clear that
$T$ is an injective linear operator, such that
\[
T(\ell_{1})\subset\left(  \mathcal{K}(E;F)\diagdown\Pi_{p}(E;F)\right)
\cup\{0\}.
\]
And therefore $\left(  \mathcal{K}(E;F)\diagdown\Pi_{p}(E;F)\right)
\cup\{0\}$ contains a vector space with the same dimension of $\ell_{1}$ (and
it is well-known that $\dim\ell_{1}=\aleph_{1}$).

Now, suppose that $F$ contains a subspace $G$ with unconditional basis
$\{e_{n};n\in\mathbb{N}\}$ with unconditional basis constant $\varrho$. Still
considering the subsets $(A_{n})$ of $\mathbb{N}$ as above, define $F_{j}$ as
the closed span of $\{e_{n};n\in A_{j}\}$ and let $P_{j}\colon
G\longrightarrow F_{j}$ be the corresponding projections. Proceeding as above
we conclude that $\Vert P_{j}\Vert\leq\varrho$. From \cite[Theorem]{davis} we
know that for each $j$ there is an operator%
\[
u_{j}\colon E\longrightarrow F_{j}%
\]
belonging to $\mathcal{K}(E;F_{j})\diagdown\Pi_{p}(E;F_{j}).$ Now by
$\widetilde{u_{j}}$ we mean the composition of $u_{j}$ with the inclusion from
$F_{j}$ to $F$. Once again consider the map
\begin{align*}
T &  :\ell_{1}\rightarrow\mathcal{K}(E;F)\\
T((a_{n})_{n=1}^{\infty}) &  =%
{\displaystyle\sum\limits_{j=1}^{\infty}}
a_{j}\widetilde{u_{j}}.
\end{align*}
Since the projections $P_{i}\colon G\longrightarrow F_{i}$ are continuous and
have norm $\leq\varrho$, it follows that
\begin{equation}
\left\Vert T\left(  (a_{n})_{n=1}^{\infty}\right)  (x)\right\Vert \geq
\rho^{-1}\left\Vert \alpha_{j}\widetilde{u_{j}}(x)\right\Vert \label{xx}%
\end{equation}
for every $j\in\mathbb{N}$. It is clear that $T$ is a linear and injective. It
also follows from (\ref{xx}) that
\[
T(\ell_{1})\subset\left(  \mathcal{K}(E;F)\diagdown\Pi_{p}(E;F)\right)
\cup\{0\}.
\]

\end{proof}

\begin{remark}
It is not difficult to show that
\[
\dim\mathcal{L}(\ell_{p};\ell_{q})=\aleph_{1}%
\]
so, for example, for $E=\ell_{p}$ $(p>1)$ and $F=\ell_{q}$ the result of the
previous theorem is optimal, i.e., we cannot improve the result to $\mu
$-lineable for $\mu>\aleph_{1}$.
\end{remark}

\section{Lineability of the set of norm attaining-operators}

Next we show that the same idea of the proof of Theorem \ref{aaa} can be
adapted to extend a result from \cite{eduardo} concerning norm-attaining operators.

In what follows $\mathcal{N\!A}^{x_{0}}(E;F)$ denotes the set of continuous
linear operators from $E$ to $F$ that attain their norms at $x_{0}$.

\begin{proposition}
\label{Lineability of NA} Let $E$ and $F$ be Banach spaces so that $E$
contains an isometric copy of $\ell_{q}$ for some $1\leq q<\infty,$ and let
$x_{0}\in\mathbb{S}_{E}.$ Then $\mathcal{N\!A}^{x_{0}}(E;F)$ is $\aleph_{1}%
$-lineable in $\mathcal{L}(E;F)$.\bigskip
\end{proposition}

\begin{proof}
The beginning of the proof follows the lines of the similar result from
\cite{eduardo}. It suffices to prove for $F=\ell_{q}$. We can write the set of
positive integers $\mathbb{N}$ as
\[
\mathbb{N=}%
{\displaystyle\bigcup\limits_{k=1}^{\infty}}
A_{k},
\]
where each
\begin{equation}
A_{k}:=\{a_{1}^{(k)}<a_{2}^{(k)}<...\}\label{hhggpp}%
\end{equation}
has the same cardinality as $\mathbb{N}$ and the sets $A_{k}$ are pairwise
disjoint. For each positive integer $k$, we define%
\[
\ell_{q}^{(k)}:=\left\{  x\in\ell_{q}:x_{j}=0\text{ if }j\notin A_{k}\right\}
.
\]
For each $k$ we can find operators $u^{(k)}$ on $\mathcal{N\!A}^{x_{0}}%
(E;\ell_{q}^{(k)})$ . By composing these operators with the inclusion of
$\ell_{q}^{(k)}$ into $\ell_{q}$ we get a vector (and we maintain the same
notation for the sake of simplicity) on $\mathcal{N\!A}^{x_{0}}(E;\ell_{q})$.
Consider the map
\begin{align*}
T &  :\ell_{1}\rightarrow\mathcal{N\!A}^{x_{0}}(E;\ell_{q})\\
T((a_{n})_{n=1}^{\infty}) &  =%
{\displaystyle\sum\limits_{j=1}^{\infty}}
a_{j}u^{(j)}.
\end{align*}
It is clear that $T$ is linear and injective. We also have that (due the
disjoint supports of the $u^{(j)}$)%
\[
T(\ell_{1})\subset\mathcal{N\!A}^{x_{0}}(E;\ell_{q}).
\]
Since $T$ is injective, it follows that $T(\ell_{1})$ is an
infinite-dimensional space and its basis has the same cardinality of the basis
of $\ell_{1}.$ Recall that $\dim(\ell_{1})=\aleph_{1}.$
\end{proof}

\end{document}